\documentclass{article}

\usepackage[final]{mlwg_2019}

\usepackage{amsthm}
\usepackage{amsmath}
\usepackage{amssymb}
\usepackage{mathrsfs}
\usepackage{graphicx}
\usepackage{microtype}
\usepackage{graphicx}
\usepackage{amsfonts}

\usepackage[colorlinks=true, allcolors=blue]{hyperref}
\usepackage[lofdepth,lotdepth]{subfig}
\usepackage{shortcuts}
\usepackage{makecell}

\usepackage[numbers]{natbib}

\title{Refined bounds for randomized experimental design}
\author{Geovani Rizk\\PSL - Universit\'e Paris Dauphine,\\ CNRS, LAMSADE, Paris, France. \\ Huawei Noah's Ark Lab  \And Igor Colin\\Huawei Noah's Ark Lab \AND Albert Thomas\\Huawei Noah's Ark Lab \And Moez Draief\\Capgemini, Paris, France}

\begin{document}

\theoremstyle{plain}
\newtheorem{theorem}{Theorem}[section]
\newtheorem{lemma}[theorem]{Lemma}
\newtheorem{prop}[theorem]{Proposition}
\newtheorem*{corollary}{Corollary}
\theoremstyle{definition}
\newtheorem{definition}{Definition}[section]
\newtheorem{conj}{Conjecture}[section]
\newtheorem{example}{Example}[section]
\theoremstyle{remark}
\newtheorem*{remark}{Remark}
\newtheorem*{note}{Note}
\newtheorem{case}{Case}

\maketitle

\begin{abstract}
  Experimental design is an approach for selecting samples among a given set so as to obtain the best estimator for a given criterion. In the context of linear regression, several optimal designs have been derived, each associated with a different criterion: mean square error, robustness, \emph{etc}. Computing such designs is generally an NP-hard problem and one can instead rely on a convex relaxation that considers probability distributions over the samples. Although greedy strategies and rounding procedures have received a lot of attention, straightforward sampling from the optimal distribution has hardly been investigated. In this paper, we propose theoretical guarantees for randomized strategies on E and G-optimal design. To this end, we develop a new concentration inequality for the eigenvalues of random matrices using a refined version of the intrinsic dimension that enables us to quantify the performance of such randomized strategies. Finally, we evidence the validity of our analysis through experiments, with particular attention on the G-optimal design applied to the best arm identification problem for linear bandits.
\end{abstract}

\section{Introduction}
\label{sec:intro}

Experimental designs consist in the selection of the best samples or \emph{experiments} for the estimation of a given quantity. A well-known and extensively studied example is the one of the ordinary least squares (OLS) estimator in the linear regression setting. The OLS estimator being unbiased, which experiments must be chosen in a fixed pool of experiments so as to minimize its variance? In the multi-dimensional case, this is done by minimizing a scalar function of its covariance matrix and several approaches have been considered such as the minimization of the determinant, the trace or the spectral norm, respectively denoted D, A and E-optimal design (see \eg \citep{Pukelsheim2006, SagnolPhd2010}). (See Appendix 1, for more details on experimental design)

E-optimal design has been exploited in practical settings such as for biological experiments \citep{FlahertyRobustDesignNIPS2006} or for treatment versus control comparisons where useful statistical interpretations have been derived, see \eg \citep{Notz1985EOptimalDesign, Rosa2017EOptimalDesign}. Another criterion, known as G-optimal design and which minimizes the worst predicted variance, has recently been investigated in the context of \emph{best arm identification} in linear bandits \citep{SoareNIPS2014,TaoICML2018,XuAISTATS2018} where one is interested in finding the experiment with maximum linear response.


The optimization problems associated to the aforementioned optimal designs (E, A, D, G) are known to be NP-hard \citep{CivrilNPHard,WelchNPHard}. The two common approaches have been to resort to greedy strategies or convex relaxations. A greedy strategy iteratively finds the best experiment whereas solving a convex relaxation returns a discrete probability distribution over the samples. 
On the one hand, performance guarantees for greedy strategies have been obtained by exploiting supermodularity and approximate supermodularity properties of the different criteria \citep{Sagnol2013,chamon2017approximate,SoareNIPS2014}. On the other hand, for performance guarantees of randomized optimal designs, only the randomized A-optimal design has been thetoretically studied with bounds on the mean square error of the associated estimator (\citep{wang2017computationally}). 


We propose in this paper to fill the gap concerning randomized E and G-optimal designs. More precisely we study their theoretical validity by providing finite-sample confidence bounds and show with experiments that they are worth being considered in practice. The paper is organized as follows. Section \ref{sec:preliminaries} defines the main notations and recalls the problem of experimental design as well as the different optimal criteria. Section \ref{sec:convergence-analysis} presents the main results of this paper for the random strategies of E and G-optimal designs. Finally, the last section shows empirical results of the studied random strategies and an application to the best arm identification problem in linear bandits.

\section{Preliminaries}
\label{sec:preliminaries}

\subsection{Definitions and notations}
\label{sec:def}

Throughout the paper, we use small bold letters for vectors (\eg $\mathbf{x}$) and capital bold letters for matrices (\eg $\mathbf{X}$). For any $d > 0$ and any vector $\mathbf{x} \in \bbR^d$, $\| \mathbf{x} \|$ will denote the usual $\ell_2$-norm of $\mathbf{x}$. For any square matrix $\mathbf{X} \in \bbR^{d \times d}$, we denote as $\|\mathbf{X}\|$ the spectral norm of $\mathbf{X}$, that is $\| \mathbf{X} \| \triangleq \sup_{\mathbf{y}: \|\mathbf{y}\|=1} \| \mathbf{X}\mathbf{y}\|$.
We let $\lambda_{\mathrm{min}}(\mathbf{X})$ be the smallest eigenvalue of $\mathbf{X}$.
For any $1 \leq i, j \leq d$, any $\mathbf{x} \in \bbR^d$ and any matrix $\mathbf{X} \in \bbR^{d \times d}$, $[\mathbf{x}]_i$ denotes the $i$-th coordinate of vector $\mathbf{x}$, $[\mathbf{X}]_i$ the vector of the $i$-th row and $[\mathbf{X}]_{ij}$ the value at the $i$-th row and $j$-th column. Finally, we denote by $\mathbf{S}^+_d$ the cone of all $d \times d$ positive semi-definite matrices and by $\Delta_d \triangleq \{ \mu \in [0, 1]^d \text{, } \sum_{i = 1}^d [\mu]_i = 1 \}$ the simplex in $\bbR^d$.

\subsection{Experimental design for linear regression}
\label{sec:experimental-design}

Given $\mathbf{X} \in \mathbb{R}^{K\times d}$ a matrix of $K$ experiments\footnote{In the remaining of this paper, we consider a finite set of experiments; some results could be easily transposed to a continuous setting.} and  $\mathbf{y} \in \mathbb{R}^{K}$ a vector of $K$ measurements, it is assumed that there exists an unknown parameter $\theta_{\star} \in \mathbb{R}^{d}$ such that for all $k \in \{1, \dots, K\}$,  $[\mathbf{y}]_k = \theta_{\star}^{\top}\mathbf{x}_k + \varepsilon_k$ where $\mathbf{x}_k = [\mathbf{X}]_k$ and $\varepsilon_1, \dots, \varepsilon_K$ are independent Gaussian random variables with zero mean and variance $\sigma^{2}$. The ordinary least squares (OLS) estimator of the parameter $\theta_{\star}$ is given by $\hat{\theta} = \argmin_{\theta} \|\mathbf{y}-\mathbf{X}\theta\|_2^2 = (\mathbf{X}^{\top}\mathbf{X})^{-1}\mathbf{X}^{\top}\mathbf{y}$.\footnote{We assume that the experiments span $\bbR^d$ so that the matrix $\mathbf{X}^{\top}\mathbf{X}$ is non singular. If this is not the case then we may project the data onto a lower dimensional space.} This estimator is unbiased and has a covariance matrix $\mathbf{\Sigma}^{-1} = \sigma^{2} (\mathbf{X}^{\top}\mathbf{X})^{-1}$.


Experimental design \citep{Pukelsheim2006} consists in estimating $\hat{\theta}$ by selecting only the experiments that are the most statistically efficient to reduce the variance.

More formally, let $n$ be the total number of selected experiments and for all $k \in \{1, \dots, K\}$, let $n_k$ be the number of times $\mathbf{x}_k$ is chosen. We have $n_k \geq 0$ and $\sum_{i=1}^{K} n_k = n$. The covariance matrix obtained with such a design can be written as $\mathbf{\Sigma}_D^{-1} = \sigma^{2} (\sum_{k=1}^{K} n_k \mathbf{x}_k \mathbf{x}_k^{\top})^{-1}$. The Loewner order on $\mathbf{S}^+_d$ being only a partial order, minimizing $\mathbf{\Sigma}_D^{-1}$ over the cone $\mathbf{S}^+_d$ is an ill-posed problem. An optimal design is thus defined thanks to scalar properties of a matrix in $\mathbf{S}^+_d$, \ie as a solution of $\min_{n_1, ..., n_K} f(\mathbf{\Sigma}_D^{-1})$ where $f : \mathbf{S}^+_d \rightarrow \mathbb{R}$. 
The two criterion $f$ we study in this paper are:
\begin{itemize}
\item \emph{E}-\textbf{optimality} : $f_E(\mathbf{\Sigma}_D^{-1}) = \|\mathbf{\Sigma}_D^{-1}\|$.
The \emph{E}-optimal design minimizes the maximum eigenvalue of $\mathbf{\Sigma}_D^{-1}$. Geometrically it minimizes the variance ellipsoid in the direction of its diameter only.
\item \emph{G}-\textbf{optimality} : $f_G(\mathbf{\Sigma}_D^{-1}) = \max_{\mathbf{x} \in \mathcal{X}} \,\mathbf{x}^{\top}\mathbf{\Sigma}_D^{-1}\mathbf{x}$. The \emph{G}-optimal design minimizes the worst possible predicted variance.
\end{itemize}

Those two optimality criteria are NP-hard optimization problems \citep{CivrilNPHard,WelchNPHard}. However approximate solutions can be found in polynomial time by relying on greedy strategies or by relaxing the problem and looking for proportions $\mu_k \in [0,1]$ instead of integers $n_k$. By letting $\mu_k = n_k/n$, the covariance matrix $\mathbf{\Sigma}_D^{-1}$ can be written as $\mathbf{\Sigma}_D^{-1} = \sigma^{2}/n \cdot (\sum_{k=1}^{K} \mu_k \mathbf{x}_k \mathbf{x}_k^{\top})^{-1}$ and it leads us to the convex optimiation problem $\min_{\mu_1, ..., \mu_K \in [0,1]} f(\mathbf{\Sigma}_D^{-1})$ which returns a discrete probability distribution over the samples. 
For more details on experimental design and optimal design criteria see \citep{boyd2004convex,Pukelsheim2006}.

\section{Convergence analysis}
\label{sec:convergence-analysis}

In this section, we analyze the behavior of random sampling along the distribution associated to the convex relaxation discussed in Section~\ref{sec:preliminaries}, for E and G-optimal designs. Let $\mathcal{X} = \{\mathbf{x}_1, \ldots, \mathbf{x}_K \} \subseteq \bbR^d$ be the set of experiments and let $\mu_E^{\star}$ and $\mu_G^{\star}$ be the optimal distributions in $\Delta_K$ associated to the convex relaxation of such designs. For any $\mu \in \Delta_K$, we denote as $\mathbf{M}(\mu)$ the matrix $\mathbf{M}(\mu) \triangleq \sum_{k = 1}^K \mu_k \mathbf{x}_k \mathbf{x}_k^{\top}$ and $f_{G, n}^{\star} \triangleq f_G\big((n \mathbf{M}(\mu_G^{\star}))^{-1}\big)$ as the objective at the optimum $\mu_G^{\star}$ for a sample size $n$.

\begin{theorem}
  \label{thm:big-result}
  Let $\mathcal{X} = \{\mathbf{x}_1, \ldots, \mathbf{x}_K\}$ be a set of experiments and let $\mu^{\star}_E$ be the solution of the relaxation of the E-optimal design. Let $0 \leq \delta \leq 1$ and let $n$ such that
  \[
    n \geq 2L \big\|\mathbf{M}(\mu_E^{\star})^{-1}\big\| \log(d/\delta),
  \]
  where $L = \max_{x \in \mathcal{X}} \|x\|^2$.
  Then, with probability at least $1 - \delta$, one has
  \[
    f_E\big(\mathbf{S}_{E, n}^{-1}\big)  \leq \left(1 +  \frac{1}{\sqrt{\displaystyle\frac{n}{2L\|\mathbf{M}(\mu_E^{\star})^{-1}\|\log(d / \delta)} } - 1}\right) f_{E, n}^{\star},
  \]
  where $\mathbf{S}_{E, n}$ is the sum of $n$ i.i.d. random matrices drawn from $\mu_E^{\star}$.

  Similarly, let $\mu_G^{\star}$ be the solution of the relaxed G-optimal design and $\mathbf{S}_{G, n}$ the associated sample sum. One has, with probability at least $1 - 2\delta$,
  \[
    f_G\big(\mathbf{S}_{G, n}^{-1}\big)  \leq \left(1 +  \frac{L}{d} \|\mathbf{M}(\mu^{\star}_G)^{-1}\|^2 \sqrt{\frac{2\sigma^2}{n} \log(d/\delta)} + o\left(\frac{1}{\sqrt{n}}\right)\right) f_{G, n}^{\star},
  \]
  with $\sigma^2 \triangleq L^2 \sum_{k = 1}^K [\mu_{G}^{\star}]_k(1 - [\mu_{G}^{\star}]_k)$.
\end{theorem}

Theses results recover the $O(1 / \sqrt{n})$ that one would expect. In addition, this confirms that the randomized approach asymptotically converges toward the true optimum, which is not the case---theoretically---for the greedy strategy. Finally, let us note that the $o(1/\sqrt{n})$ in the G-optimal design rate depends on the interaction between Hoeffding for $\lambda_{\mathrm{min}}(\mathbf{S}_n)$ and Bennett for $\|\mathbf{S}_n - \bbE\mathbf{S}_n\|$. We refer the reader to the supplementary material in Appendix \ref{sec:chernoff} for the full bound.

\subsection{A refined approach of the dimension}
\label{subsec:intdim}


In this section, we introduce two quantities, derived from the concept of \emph{intrinsic dimension} \cite{koltchinskii2016asymptotics, minsker2011some}, that allow us to refine the convergence rate for G-optimal design.
We recall the definition of intrinsic dimension.
\begin{definition}[Intrinsic dimension]
  Let $d > 0$ and $\mathbf{S} \in \bbR^{d \times d}$ be a positive semi-definite matrix. The intrinsic dimension of $\mathbf{S}$, denoted $\intdim(\mathbf{S})$, is defined as follows:
  \[
    \intdim(\mathbf{S}) \triangleq \frac{\trace(\mathbf{S})}{\|\mathbf{S}\|} \leq d.
  \]
\end{definition}

Using this definition, one may alter the concentration results on the spectral norm, 
by replacing the dimension $d$ by $2 \times \intdim(\bbE\mathbf{S}_n)$. For a matrix with eigenvalues decreasing fast enough, the improvement may be substantial---see \cite[Ch. 7]{TroppBook} and references therein for more details. The main drawback of this definition is that if eigenvalues are all of the same order of magnitude, one will not notice a sensible improvement; this is typically the case in G-optimal design as eigenvalues are designed to be large overall. We propose a refined version of the intrinsic dimension allowing improvements even with a narrow spectrum, in the form of 2 complementary quantities.

\begin{definition}[Upper and lower intrinsic dimension]
  Let $d > 0$ and $\mathbf{S} \in \bbR^{d \times d}$ be a positive semi-definite matrix. The upper and lower intrisic dimensions of $\mathbf{S}$, denoted $\updim(\mathbf{S})$ and $\lowdim(\mathbf{S})$ respectively, are defined as follows:
  \[
    \left\{
      \begin{array}{rcl}
        \updim(\mathbf{S}) &\triangleq &\dfrac{\trace\big(\mathbf{S} - \lambda_{\mathrm{min}}(\mathbf{S}) \mathbf{I}\big)}{\|\mathbf{S}\| - \lambda_{\mathrm{min}}(\mathbf{S})}\\[1em]
        \lowdim(\mathbf{S}) &\triangleq &\dfrac{\trace\big(\| \mathbf{S} \| \mathbf{I} - \mathbf{S}\big)}{\|\mathbf{S}\| - \lambda_{\mathrm{min}}(\mathbf{S})} = d - \updim(\mathbf{S}).       
      \end{array}
    \right.
  \]
\end{definition}

These new quantities use both the largest and the smallest eigenvalues to rescale the spectrum, which is of interest in our setting. Using this definition, one is able to formulate new concentration results on random matrices, including a concentration result on the lowest eigenvalue. In this particular case however, we are more interested in the potential speed up provided for the spectral norm, since it is the value controlling the slowest term in the G-optimal design error.

\begin{theorem}
  Let $\mathcal{X} = \{\mathbf{x}_1, \ldots, \mathbf{x}_K\}$ be a set of experiments and let $\mu^{\star}_G$ be the solution of the relaxation of the G-optimal design. Let $\mathbf{X}_1, \ldots, \mathbf{X}_n$ be $n$ i.i.d.\ random matrices drawn according to $\mu^{\star}_G$ and $\mathbf{S}_n$ their sum. Let $\mathbf{V}$ be the covariance matrix of $\mathbf{X}_1$, that is $\mathbf{V} \triangleq \bbE\big[\mathbf{X}_1^2] - \mathbf{M}(\mu^{\star}_G)^2$ and let $\kappa$ be its condition number.
  
  Let $0 \leq \delta \leq 1$ and let $n$ such that
  \[
    n \geq \left(\frac{4L^2}{9\|\mathbf{V}\|}\log(\tilde{d}/\delta)\right)
  \]
  where $L = \max_{x \in \mathcal{X}} \|x\|^2$ and $\tilde{d}$ is defined by 
  $\tilde{d} \triangleq \updim(\mathbf{V}) + \lowdim(\mathbf{V})e^{-n(1 - \kappa^{-1})/16} < d$.
  Then, with probability at $1 - 2\delta$, one has
  \[
    f_G\big(\mathbf{S}_{G, n}^{-1}\big)  \leq \left(1 +  \frac{L}{d} \|\mathbf{M}(\mu^{\star}_G)^{-1}\|^2 \sqrt{\frac{4\sigma^2}{n} \log(\tilde{d}/\delta)} + o\left(\frac{1}{\sqrt{n}}\right)\right) f_{G, n}^{\star}.
  \]
\end{theorem}
We refer the reader to the supplementary material in Appendix \ref{sec:intdim} for the proof of this result. 

\section{Experiments}
\label{sec:experiments}

In this section we compare the performances of randomized E and G-optimal designs against their greedy counterparts. We first show the behavior of the randomized E-optimal design on a synthetic data set. We then apply the randomized G-optimal design to the problem of best arm identification and compare it to the greedy approach used in \citep{SoareNIPS2014}. We refer the reader to Appendix \ref{sec:mab-linear} and \ref{sec:experiment_details} for more details on best arm identification for linear bandits and on the experiments setting, respectively.

\begin{figure}[h]
  \begin{center}
    \subfloat[Evolution of score with $n$]{
      \includegraphics[width=0.3\textwidth]{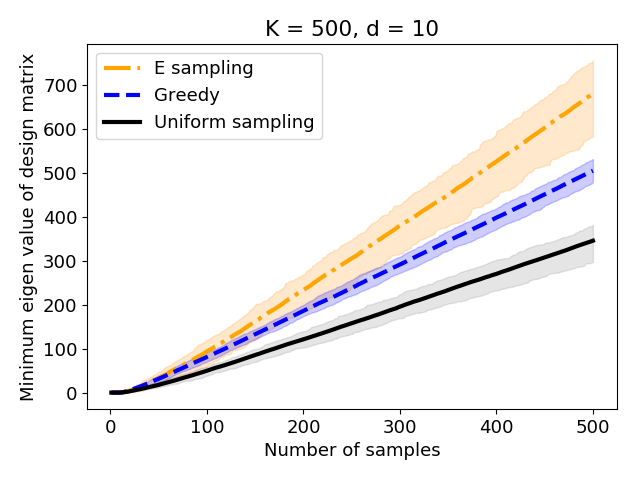}
      \label{e_random_vs_greedy_iterations}
    }
    \subfloat[Evolution of score with $d$]{
      \includegraphics[width=0.3\textwidth]{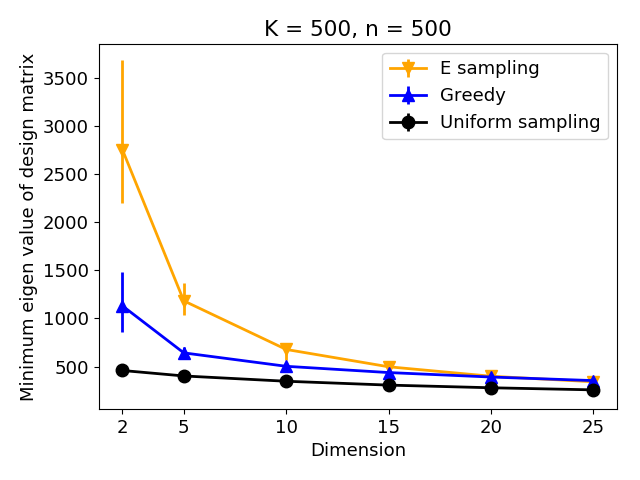}
      \label{e_random_vs_greedy_dimension}
    }
    \subfloat[Evolution of score with $d$]{
      \includegraphics[width=0.3\textwidth]{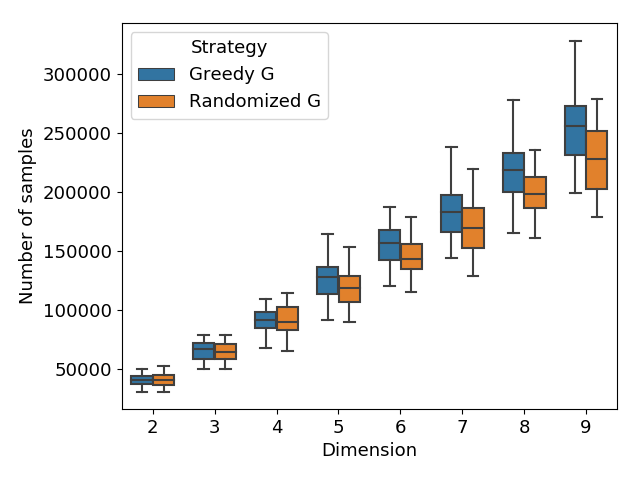}
      \label{expe_soare_random_vs_greedy}
    }
    \label{fig:e_random_vs_greedy}
  \end{center}
\end{figure}

\section{Conclusion}
\label{sec:conclusion}

We have shown the convergence of randomized scheme for G and E-optimal criteria at a rate of $O(1/\sqrt{n})$. We also evidenced the dependence of the rate in a specific characteristic of the covariance matrix for the sampling. Empirically, the random sampling enjoys a favorable comparison with the greedy approach, even in the bandit application. One possible extension of this work could be to investigate the setting of batch or parallel bandits, using a random sampling to select a batch of arms before observing the rewards.

\clearpage
{\small 
\bibliographystyle{plain}
\bibliography{biblio}
}

\clearpage
\appendix

\section{Chernoff inequalities on matrices}
\label{sec:chernoff}

Many concentration inequalities have been developped for bounding the deviation of a sum of i.i.d.\ random variables. In particular, Chernoff inequalities have been extensively studied and derived due to their exponential decay rate on tail distributions. Here we show how these bounds can be extended to random matrices (see \eg \cite{tropp2015introduction} for an introduction on that matter).

\subsection{Additional notations}

For any Hermitian matrices $\mathbf{X}, \mathbf{Y}$, we write $\mathbf{X} \preceq \mathbf{Y}$ if and only if the matrix $\mathbf{Y} - \mathbf{X}$ is positive semidefinite. Recall that for any Hermitian matrix $\mathbf{X}$, there exists a unitary matrix $\mathbf{P}$ and a diagonal matrix $\mathbf{D}$ such that $\mathbf{X} = \mathbf{P} \mathbf{D} \mathbf{P}^{\top}$. For such a matrix and for any function $f : \bbR \rightarrow \bbR$, we denote as $f(\mathbf{X})$ the extension of $f$ to a Hermitian matrix, defined as follows:
\[
  f(\mathbf{X}) \triangleq \mathbf{P}
  \begin{pmatrix}
    f([\mathbf{D}]_{11}) &&\\
    &\ddots&\\
    &&f([\mathbf{D}]_{dd})
  \end{pmatrix}
  \mathbf{P}^{\top}.
\]
In particular, for any scalar $x \in \bbR$, we define $(x)_+ \triangleq \max(x, 0)$ so $(\mathbf{X})_+$ is the projection of $\mathbf{X}$ onto the positive semidefinite cone. We will use the exponential function for both scalars and matrices: for the sake of clarity, we denote as $e^x$ the exponential of a scalar and $\exp(\mathbf{X})$ the exponential of a matrix. We will denote as $\Sp(\mathbf{X})$ the spectrum of $\mathbf{X}$, that is the set of all eigenvalues associated to $\mathbf{X}$. The identity matrix and the zero matrix in dimension $d$ are denoted $\mathbf{I}_d$ and $\mathbf{0}_d$, respectively; when clear from context, we drop the $d$ index.

\subsection{Useful lemmas}
\label{subsec:useful-lemmas}

Before stating the concentration inequalities of interest, we need to state several useful lemmas. These lemmas are key for proving concentration of random matrices, as we need similar guarantees for matrix ordering ($\preceq$) than for scalar ordering ($\leq$). We first state two lemmas that will ensure order preserving under basic operations.

\begin{lemma}[Conjugation Rule]
  \label{lma:conjugation-rule}
  Let $\mathbf{M}, \mathbf{N} \in \bbR^{d \times d}$ be two Hermitian matrices, such that $\mathbf{M} \preceq \mathbf{N}$. Let $p > 0$ and let $\mathbf{Q} \in \bbR^{p \times d}$. Then, one has
  \[
    \mathbf{Q} \mathbf{M} \mathbf{Q}^{\top} \preceq \mathbf{Q} \mathbf{N} \mathbf{Q}^{\top}.
  \]
\end{lemma}
\begin{proof}
  The proof is immediate when considering $\mathbf{Q}( \mathbf{N} - \mathbf{M} )\mathbf{Q}^{\top}$ and using the definition of a positive semidefinite matrix.
\end{proof}

\begin{lemma}[Transfer Rule]
  \label{lma:transfer-rule}
  Let $\mathbf{M} \in \bbR^{d \times d}$ be a Hermitian matrix and let $f, g : \bbR \rightarrow \bbR$ be such that, for any $x \in \Sp(\mathbf{M})$, $f(x) \leq g(x)$. Then, one has
  \[
    f(\mathbf{M}) \preceq g(\mathbf{M}).
  \]
\end{lemma}
\begin{proof}
  Let $\mathbf{D}$ be the diagonal matrix in the spectral decomposition of $\mathbf{M}$. Since $f \leq g$ on $\Sp(\mathbf{M})$, one has $f(\mathbf{D}) \preceq g(\mathbf{D})$. The conjugation rule then allows us to conclude.
\end{proof}

Finally, we state two lemmas ensuring that two more complex operations ($\trace\exp$ and $\log$, respectively) preserve the order. Please note that this is usually not the case, even for operators that are monotone on $\bbR$---\eg the exponential does not preserve the order. 
\begin{lemma}[Monotonicity of the trace of the exponential]
  \label{lma:trace-monotonicity}
  Let $\mathbf{M}, \mathbf{N} \in \bbR^{d \times d}$ be two Hermitian matrices such that $\mathbf{M} \preceq \mathbf{N}$. Then for any non-decreasing function $\psi : \bbR \rightarrow \bbR$, one has:
  \[
    \trace\big( \psi(\mathbf{M}) \big) \leq \trace\big( \psi(\mathbf{N})\big).
  \]
  In particular,
  \[
    \trace \exp(\mathbf{M}) \leq \trace \exp(\mathbf{N}).
  \]
\end{lemma}
\begin{proof}
  Let $\lambda_1(\mathbf{M}) \geq \ldots \geq \lambda_d(\mathbf{M})$ and $\lambda_1(\mathbf{N}) \geq \ldots \geq \lambda_d(\mathbf{N})$ be the sorted eigenvalues of $\mathbf{M}$ and $\mathbf{N}$, respectively. Then, for $1 \leq i \leq d$, one can define an eigenvalue as follows:
  \[
    \lambda_i(\mathbf{M}) = \max_{\mathbb{L} \subseteq \bbR^d:\dim \mathbb{L} = i}\; \min_{\mathbf{u} \in \mathbb{L}: \|\mathbf{u}\| = 1} \mathbf{u}^{\top} \mathbf{M} \mathbf{u}.
  \]
  Using the fact that $\mathbf{M} \preceq \mathbf{N}$, one can deduce that for any $1 \leq i \leq d$, $\lambda_i(\mathbf{M}) \leq \lambda_i(\mathbf{N})$.

  Since $\psi$ is a non-decreasing function on $\bbR$, one has that for any $1 \leq i \leq d$, $\psi\big(\lambda_i(\mathbf{M})\big) \leq \psi\big(\lambda_i(\mathbf{N})\big)$. Summing over the dimensions leads to the desired result.
\end{proof}

\begin{lemma}[Monotonicity of the logarithm]
  Let $\mathbf{M}, \mathbf{N} \in \bbR^{d \times d}$ be two positive definite matrices such that $\mathbf{M} \preceq \mathbf{N}$. Then one has:
  \[
    \log(\mathbf{M}) \preceq \log(\mathbf{N}).
  \]
\end{lemma}
\begin{proof}
  We will first prove that for any $\gamma \in \bbR_+$, $(\mathbf{M} + \gamma\mathbf{I})^{-1} \succeq (\mathbf{N} + \gamma\mathbf{I})^{-1}$.

  The facts that $\mathbf{M} \preceq \mathbf{N}$ and $\gamma \geq 0$ imply that $\mathbf{M} + \gamma \mathbf{I} \preceq \mathbf{N} + \gamma \mathbf{I}$. Using Lemma~\ref{lma:conjugation-rule}, we obtain:
  \[
    \mathbf{0} \prec (\mathbf{N} + \gamma \mathbf{I})^{-1/2} (\mathbf{M} + \gamma \mathbf{I}) (\mathbf{N} + \gamma \mathbf{I})^{-1/2} \preceq \mathbf{I}.
  \]
  Taking the inverse yields:
  \[
    (\mathbf{N} + \gamma \mathbf{I})^{1/2} (\mathbf{M} + \gamma \mathbf{I})^{-1} (\mathbf{N} + \gamma \mathbf{I})^{1/2} \succeq \mathbf{I}.
  \]
  Finally, applying again Lemma~\ref{lma:conjugation-rule} with $(\mathbf{N} + \gamma \mathbf{I})^{-1/2}$ yields:
  \[
    (\mathbf{M} + \gamma \mathbf{I})^{-1} \succeq (\mathbf{N} + \gamma \mathbf{I})^{-1}.
  \]

  Let us now focus on the main result. First recall that the logarithm of a positive scalar can be expressed using its integral representation, that is
  \[
    \log x = \int_{0}^{+\infty} \left(\frac{1}{1 + t} - \frac{1}{x + t}\right)\mathrm{d}t,
  \]
  for any $x > 0$. Therefore, the logarithm of a matrix $\mathbf{X} \succ \mathbf{0}$ can be expressed similarly:
  \[
    \log \mathbf{X} = \int_{0}^{+\infty} \left( \frac{1}{1 + t} \mathbf{I} - (\mathbf{X} + t\mathbf{I})^{-1} \right) \mathrm{d}t.
  \]
  In the beginning of the proof, we have shown that for any $\gamma \geq 0$, $(\mathbf{M} + t\mathbf{I})^{-1} \succeq (\mathbf{N} + t\mathbf{I})^{-1}$. Therefore, one has:
  \[
    \frac{1}{1 + \gamma} \mathbf{I} - (\mathbf{M} + \gamma\mathbf{I})^{-1} \preceq \frac{1}{1 + \gamma} \mathbf{I} - (\mathbf{N} + \gamma\mathbf{I})^{-1},
  \]
  and integrating over $\gamma$ yields the final result.
\end{proof}

\subsection{Chernoff inequalities}
\label{subsec:chernoff}

Let $n > 0$ and let $\mathbf{X}_1, \ldots, \mathbf{X}_n$ be i.i.d.\ positive semidefinite matrices, such that there exists $L > 0$ verifying:
\[
  \mathbf{0} \preceq \mathbf{X}_1 \preceq L \mathbf{I},
\]
almost surely.
Let us now consider the random matrix $\mathbf{S} = \sum_{i = 1}^n \mathbf{X}_i$. In what follows, we will develop Chernoff bounds in order to control both $\| \mathbf{S} \| / \| \bbE\mathbf{S} \|$ and $\| \mathbf{S} - \bbE \mathbf{S} \|$.

In the scalar case, Chernoff's bounds for the sum of independent variables are based on the fact that the exponential converts a sum into a products, that is for $n$ i.i.d.\ random variables $X_1, \ldots, X_n$:
\[
  \bbE e^{\sum_{i = 1}^n X_i} = \bbE \prod_{i = 1}^n e^{X_i},
\]
and then one uses the independence to pull the product out of the expectation. For two symmetric matrices $\mathbf{M}, \mathbf{N} \in \bbR^{d \times d}$ however, the relation $\exp(\mathbf{M} + \mathbf{N}) = \exp\mathbf{M} \exp\mathbf{N}$ does not hold in general---it holds if the matrices commute. Hopefully, the following theorem gives us a way to overcome this issue.
\begin{theorem}
  \label{thm:master-bound}
  Let $n, d > 0$ and let $\mathbf{X}_1, \ldots, \mathbf{X}_n$ be i.i.d.\ symmetric matrices in $\bbR^d$. Then, for any $t \in \bbR$, one has
  \[
    \bbP\left(\left\|\sum_{i = 1}^n \mathbf{X}_i\right\| \geq t \right) \leq \inf_{\eta > 0} e^{-\eta t} \trace \exp\left(\sum_{i = 1}^n \log \bbE \exp(\eta \mathbf{X}_i) \right),
  \]
  and similarly
  \[
    \bbP\left(\lambda_{\mathrm{min}}\left(\sum_{i = 1}^n \mathbf{X}_i\right) \leq t \right) \leq \inf_{\eta < 0} e^{-\eta t} \trace \exp\left(\sum_{i = 1}^n \log \bbE \exp(\eta \mathbf{X}_i) \right).
  \]
\end{theorem}
\begin{proof}
  We start by the first inequality. Let $t \in \bbR$ and let $\eta > 0$. As in the scalar case, one has:
  \[
    \bbP\left(\left\|\sum_{i = 1}^n \mathbf{X}_i\right\| \geq t \right) = \bbP \left( e^{\eta \|\sum_{i = 1}^n \mathbf{X}_i \|} \geq e^{\eta t} \right) \leq e^{-\eta t} \bbE e^{\eta \|\sum_{i = 1}^n \mathbf{X}_i \|},
  \]
  where the last inequality is an application of Markov's inequality. Using the fact that for a positive semidefinite matrix $\mathbf{X}$, $\| \mathbf{X} \| \leq \trace \mathbf{X}$, we obtain:
  \[
    \bbE e^{\eta \| \sum_{i = 1}^n \mathbf{X}_i \|} = \bbE \left\| \exp\left(\eta \sum_{i = 1}^n \mathbf{X}_i\right) \right\| \leq \bbE \trace \exp\left(\eta \sum_{i = 1}^n \mathbf{X}_i\right).
  \]
  We will now use Lieb's Theorem \cite{lieb1973convex}, which states that for any symmetric matrix $\mathbf{A} \in \bbR^{d \times d}$, the mapping $\mathbf{M} \mapsto \trace \exp(\mathbf{A} + \log\mathbf{M})$ is concave on the cone of positive semidefinite matrices. This allows us to bound the above term as follows:
  \begin{align*}
    \bbE \trace \exp\left(\eta \sum_{i = 1}^n \mathbf{X}_i\right)
    &= \bbE \bbE\left[\trace \exp\left(\eta \sum_{i = 1}^n \mathbf{X}_i\right) \Bigg| \mathcal{F}_{n - 1}\right]\\
    &= \bbE \bbE\left[\trace \exp\left(\eta \sum_{i = 1}^{n - 1} \mathbf{X}_i + \log \exp (\eta \mathbf{X}_n) \right) \Bigg| \mathcal{F}_{n - 1}\right]\\
    &\leq \bbE \trace \exp\left(\eta \sum_{i = 1}^{n - 1} \mathbf{X}_i + \log \bbE \exp (\eta \mathbf{X}_n) \right).
  \end{align*}
  Iterating over $n$ yields:
  \[
    \bbE \trace \exp\left(\eta \sum_{i = 1}^n \mathbf{X}_i\right) \leq \trace \exp \left(\sum_{i = 1}^n \log \bbE \exp(\eta \mathbf{X}_i) \right),
  \]
  hence the result.

  The second inequality is a direct consequence from the fact that for any $\eta < 0$ and any matrix $\mathbf{X}$, $\eta \lambda_{\mathrm{min}}(\mathbf{X}) = \| \eta \mathbf{X} \|$.
\end{proof}

The formulation of Theorem~\ref{thm:master-bound}, although more complicated than is the scalar case, is very helpful for matrix concentration analysis. Indeed, since $\trace \exp$ and $\log$ are both order-preserving operators on positive matrices, a bound on $\bbE \exp(\eta \mathbf{X}_1)$ will now be enough to provide an overall bound of the extreme eigenvalues.

\subsubsection{Hoeffding's inequality}
\label{subsubsec:hoeffding}

The bound we develop here ensures that $\|\mathbf{S}\|$ and $\lambda_{\mathrm{min}}(\mathbf{S})$ do not deviate too much from their counterpart on $\bbE \mathbf{S}$.
We are now ready to state the first result.
\begin{theorem}
  \label{thm:chernoff-matrix}
  Let $\mathbf{X}_1, \ldots, \mathbf{X}_n$ be i.i.d.\ positive semidefinite random matrices, such that there exists $L > 0$ verifying $\mathbf{0} \preceq \mathbf{X}_1 \preceq L \mathbf{I}$. Let $\mathbf{S}$ be defined as:
  \[
    \mathbf{S} \triangleq \sum_{i = 1}^n \mathbf{X}_i.
  \]
  Then, for any $0 < \varepsilon < 1$, one can lowerbound $\lambda_{\mathrm{min}}(\mathbf{S})$ as follows:
  \[
    \bbP(\lambda_{\mathrm{min}}(\mathbf{S}) \leq (1 - \varepsilon) \lambda_{\mathrm{min}}(\bbE \mathbf{S})) \leq
    d \left(\frac{e^{-\varepsilon}}{(1 - \varepsilon)^{1 - \varepsilon}}\right)^{\frac{n\lambda_{\mathrm{min}}(\bbE \mathbf{X}_1)}{L}}.
  \]
  Similarly, one can upperbound $\| \mathbf{S} \|$ as follows:
  \[
    \bbP(\|\mathbf{S}\| \geq (1 + \varepsilon) \|\bbE \mathbf{S} \| )\leq
    d \left(\frac{e^{\varepsilon}}{(1 + \varepsilon)^{1 + \varepsilon}}\right)^{\frac{n\|\bbE \mathbf{X}_1\|}{L}}.
  \]
\end{theorem}

The following corollary shows an alternate (but slightly weaker) formmulation which is closer to usual concentration results.

\begin{corollary}
  Let $\mathbf{X}_1, \ldots, \mathbf{X}_n$ and $\mathbf{S}$ be defined as above. For any $0 < \varepsilon < 1$, one can lowerbound $\lambda_{\mathrm{min}}(\mathbf{S})$ as follows:
  \[
    \bbP(\lambda_{\mathrm{min}}(\mathbf{S}) \leq (1 - \varepsilon) \lambda_{\mathrm{min}}(\bbE \mathbf{S})) \leq
    d \exp\left(-\frac{\varepsilon^2 \lambda_{\mathrm{min}}(\bbE \mathbf{X}_1)}{2L}\right).
  \]
\end{corollary}

Before proving Theorem~\ref{thm:chernoff-matrix}, we state a useful lemma for bouding moment generating function of random positive semi-definite matrices.

\begin{lemma}
  \label{lma:mgf-bound}
  Let $t \in \bbR$ and let $\mathbf{X}$ be a random matrix such that $\mathbf{0} \preceq \mathbf{X} \preceq L \mathbf{I}$ almost surely for some $L \geq 0$. Then, one has:
  \[
    \bbE \exp(t \mathbf{X}) 
    \preceq \mathbf{I} + \frac{e^{tL} - 1}{L} \bbE \mathbf{X}
    \preceq \exp\left(\frac{e^{tL} - 1}{L} \bbE \mathbf{X}\right).
  \]
\end{lemma}
\begin{proof}
  Both inequalities are derived from the convexity of the exponential. We will write scalar inequalities based on convexity and then extend them to matrices using the transfer rule in Lemma~\ref{lma:transfer-rule}. Let $t \in \bbR$, for any $0 \leq x \leq L$, the following holds:
  \[
    e^{tx} \leq e^{0} + \frac{x}{L} (e^{tL} - e^{0}) = 1 + \frac{e^{tL} - 1}{L} x.
  \]
  Since $\mathbf{0} \preceq \mathbf{X} \preceq L \mathbf{I}$ almost surely, this can be extented to the matrix exponential using the transfer rule in Lemma~\ref{lma:transfer-rule}:
  \[
    \exp(t \mathbf{X}) \preceq \mathbf{I} + \frac{e^{tL} - 1}{L} \mathbf{X}.
  \]
  Taking the expectation yields the result:
  \[
    \bbE \exp(t \mathbf{X}) \preceq \mathbf{I} + \frac{e^{tL} - 1}{L} \bbE \mathbf{X}.
  \]
  The second inequality is also an application of Lemma~\ref{lma:transfer-rule} using the inequality $1 + x \leq e^x$ for any $x \in \bbR$.
\end{proof}

\begin{proof}[Proof of Theorem~\ref{thm:chernoff-matrix}]
  Let $t > 0$. Combining Lemma~\ref{lma:mgf-bound} and Theorem~\ref{thm:master-bound} yields:
  \[
    \bbP\left(\lambda_{\mathrm{min}}(\mathbf{S}) \leq t \right) \leq \inf_{\eta < 0} e^{-\eta t}\trace \exp\left(\frac{e^{\eta L} - 1}{L} \bbE \mathbf{S}\right).
  \]
  Reintegrating $\| \cdot \|$ into the RHS yields:
  \begin{equation}
    \inf_{\eta < 0} e^{-t\eta}\trace \exp\left(\frac{e^{\eta L} - 1}{L} \bbE \mathbf{S}\right)
    \leq \inf_{\eta < 0} d e^{-t\eta} \left\|\exp\left(\frac{e^{\eta L} - 1}{L} \bbE \mathbf{S}\right)\right\|
    = \inf_{\eta < 0} d e^{-t\eta + \frac{e^{\eta L} - 1}{L} \lambda_{\mathrm{min}}(\bbE \mathbf{S})}. \label{ineq:lmin-almost-final}
  \end{equation}
  Since the inequality holds for any $\eta < 0$, we can optimize over $\eta$. The optimal (lowest) value is reached for $\eta = (L)^{-1} \log(t / \lambda_{\mathrm{min}}(\bbE\mathbf{S}))$, which is negative if and only if $t < \lambda_{\mathrm{min}}(\bbE\mathbf{S})$. Let us make the change of variable $t = (1 - \varepsilon) \lambda_{\mathrm{min}}(\bbE\mathbf{S})$ for $0 < \varepsilon < 1$, so the condition holds. Substituting the value of $\eta$ into~\eqref{ineq:lmin-almost-final} yields:
  \[
    \bbP\left(\lambda_{\mathrm{min}}(\mathbf{S}) \leq (1 - \varepsilon)\lambda_{\mathrm{min}}(\bbE \mathbf{S})\right)
    \leq d e^{\big((\varepsilon-1)\log(1 - \varepsilon) - \varepsilon\big) \frac{n\lambda_{\mathrm{min}}(\bbE \mathbf{X}_1)}{L}},
  \]
  and the result holds.
\end{proof}

\begin{remark}
  Without additional characterization of the problem, the bound $\bbE\|\mathbf{X}\| \leq \trace \bbE \mathbf{X} \leq d \|\bbE \mathbf{X}\|$ is tight: consider a diagonal random matrix $\mathbf{X}$ such that for any $1 \leq i \leq d$, $\bbP(\mathbf{X} = \mathbf{E}_{ii}) = 1/d$, where $(\mathbf{E}_{ii})_{1 \leq i \leq d}$ are the diagonal elements from the canonical base. Then, $\|\mathbf{X}\| =  1$ and $\|\bbE \mathbf{X}\| = 1/d$, so $\|\mathbf{X}\| = d \|\bbE \mathbf{X}\|$. If we consider the best arm identification application, this case essentially boils down to the MAB setting and would make the whole linear modeling irrelevant: maybe there is a more subtle way of characterizing linear bandits in order to avoid a brutal $d$ factor in the bound.
\end{remark}

\subsubsection{Bennett's and Bernstein's inequalities}
\label{subsubsec:bernstein}

Using the Chernoff's bound, we were able to prove, with high probability, the following:
\[
  \mathbf{x}^{\top}\left(n \sum_{k = 1}^K [\mu^{\star}_G]_k \mathbf{x}_k \mathbf{x}_k^{\top} \right)^{-1} \mathbf{x}
  \leq \mathbf{x}^{\top}\left(\sum_{i = 1}^n \mathbf{X}_i\right)^{-1} \mathbf{x}
  \leq \frac{\|\mathbf{x}\|^2}{(1 - \varepsilon)} \lambda_{\mathrm{max}}\left(n\sum_{k = 1}^K [\mu^{\star}_G]_k \mathbf{x}_k \mathbf{x}_k^{\top} \right)^{-1}.
\]
This is not enough to ensure the convergence of the randomized sampling. Considering again the random matrix
\[
  \mathbf{S}_n = \sum_{i = 1}^n \mathbf{X}_i,
\]
our goal is to bound the following quantity:
\[
  \mathbf{x}^{\top}\big(\mathbf{S}_n^{-1} - (\bbE\mathbf{S}_n)^{-1}\big)\mathbf{x}.
\]
One way to bound the above quantity is to bound the maximum eigenvalue of $\mathbf{S}_n^{-1} - (\bbE \mathbf{S}_n)^{-1}$. One has:
\[
\mathbf{S}_n^{-1} - (\bbE \mathbf{S}_n)^{-1} = \mathbf{S}_n^{-1}\big( \mathbf{I} - \mathbf{S}_n (\bbE \mathbf{S}_n)^{-1}\big) = \mathbf{S}_n^{-1}( \bbE\mathbf{S} - \mathbf{S}_n)(\bbE \mathbf{S}_n)^{-1}.
\]
In Section~\ref{subsubsec:hoeffding}, we used Hoeffding's inequality to upperbound $\big\|\mathbf{S}_n^{-1}\big\|$ based on $\big\|(\bbE\mathbf{S}_n)^{-1}\big\|$ value. Therefore, we only need to care about the central term. Since the random matrix $\bbE\mathbf{S}_n - \mathbf{S}_n$ is not necessarily positive semi-definite anymore, we cannot use Hoeffding's inequality. We can use Bernstein's inequality however, as stated in the following theorem.
\begin{theorem}
  \label{thm:bennett}
  Let $\mathbf{X}_1, \ldots, \mathbf{X}_n$ be $n$ i.i.d.\ random symmetric matrices such that $\bbE \mathbf{X}_1 = \mathbf{0}$ and there exists $L > 0$ such that $\| \mathbf{X}_1 \| \leq L$, almost surely. Let $\mathbf{S}_n \triangleq \sum_{i = 1}^n \mathbf{X}_i$. Then, for any $t > 0$, one has:
  \[
    \bbP \left( \left\| \mathbf{S}_n \right\| \geq t \right) \leq d e^{\displaystyle-\frac{t^2}{2Lt/3 + 2n\sigma^2}},
  \]
  where $\sigma^2 \triangleq \left\| \bbE \big[\mathbf{X}_1^2\big] \right\|$.
\end{theorem}

As in the scalar case, Bernstein's inequality relies on using the Taylor expansion of the exponential to bound the moment generating function, so we will need the following lemma.

\begin{lemma}
  \label{lma:bennett}
  Let $L > 0$ and let $\mathbf{X}$ be a random Hermitian matrix such that $\bbE \mathbf{X} = \mathbf{0}$ and $\mathbf{X} \preceq L \mathbf{I}$ almost surely. Then, for any $0 < t < 3 / L$, one has:
  \[
    \bbE \exp (t\mathbf{X}) \preceq \exp\left(\frac{t^2/2}{1 - tL/3} \bbE \big[\mathbf{X}^2 \big]\right).
  \]
\end{lemma}

\begin{proof}
  Similarly to the Hoeffding's case, we will show a result for the exponential of a scalar and extend it to a Hermitian matrix. Let $L > 0$, $0 < x < L$ and $0 < t < 3 / L$. Let us define $f : [0, L] \rightarrow \bbR$ such that for any $0 < y < L$,
  \[
    f(y) \triangleq \frac{e^{ty} - ty - 1}{y^2}.
  \]
  In particular, one has $e^{tx} = 1 + tx + x^2 f(x)$. Notice that $f$ is increasing, so $e^{tx} \leq 1 + tx + x^2 f(L)$. Now, using the Taylor expansion of the exponential, we can write:
  \[
    f(L)
    = \frac{e^{tL} - tL - 1}{L^2}
    = \frac{1}{L^2} \sum_{k \geq 2} \frac{(tL)^{k}}{k!}
    \leq \frac{t^2}{2} \sum_{k \geq 2} \frac{(tL)^{k - 2}}{3^{k-2}}
    = \frac{t^2/2}{1 - tL/3},
  \]
  where the inequality comes from the fact that $k! \geq 2 \times 3^{k - 2}$, for any $k \geq 2$.

  Now, using the fact that $\mathbf{X} \preceq L \mathbf{I}$ almost surely, we can obtain the following bound:
  \[
    \exp(t\mathbf{X})
    \preceq \mathbf{I} + t \mathbf{X} + \mathbf{X} (f(L) \mathbf{I}) \mathbf{X}
    = \mathbf{I} + t \mathbf{X} + f(L) \mathbf{X}^2.
  \]
  Finally, taking the expectation and combining this result with a common bound of the exponential, we obtain:
  \[
    \bbE \exp(t\mathbf{X})
    \preceq \mathbf{I} + \frac{t^2 / 2}{1 - tL/3} \bbE\big[\mathbf{X}^2\big]
    \preceq \exp\left(\frac{t^2 / 2}{1 - tL/3} \bbE\big[\mathbf{X}^2\big] \right),
  \]
  hence the result.
\end{proof}

We are now ready to prove Bernstein's inequality for matrices.

\begin{proof}[Proof of Theorem~\ref{thm:bennett}]
  Let $0 < \eta < 3 / L$, using Markov's inequality one has:
  \[
    \bbP(\|\mathbf{S}_n\| \geq t)
    = \bbP \left( e^{\eta \|\mathbf{S}_n\|} \geq e^{\eta t}\right)
    \leq e^{-\eta t} \bbE e^{\eta \|\mathbf{S}_n\|}.
  \]
  By Lemma~\ref{lma:bennett} and the subadditivity of the matrix CGF \cite[Lemma 3.5.1, Ch. 3]{tropp2015introduction}, we obtain
  \[
    \bbE e^{\eta \|\mathbf{S}_n\|}
    = \bbE \|\exp(\eta \mathbf{S}_n)\|
    \leq \trace \bbE \exp(\eta \mathbf{S}_n)
    \leq \trace \exp\left(\frac{t^2 / 2}{1 - tL/3} \bbE\big[\mathbf{S}_n^2\big] \right).
  \]
  Plugin the trace back into the exponential yields:
  \[
    \bbE e^{\eta \| \mathbf{S}_n \|} \leq \trace \exp\left(\frac{\eta^2 / 2}{1 - \eta L/3} \bbE\big[\mathbf{S}_n^2\big] \right) \leq d e^{\frac{\eta^2 / 2}{1 - \eta L/3} \left\|\bbE\big[\mathbf{S}_n^2\big]\right\|}.
  \]
  Optimizing on $\eta$ would lead to a complicated result, so we use instead $\eta = t / (n\sigma^2 + tL/3)$, which verifies the condition $\eta < 3 / L$ and yields the final result. 
\end{proof}

The relationship between the precision ($nt$ in Theorem~\ref{thm:bennett}) and the confidence level $\delta_b$ (the RHS of the concentration inequality) is more complicated in Bernstein's inequality than in Hoeffding's. It requires solving a second order polynomial equation and leads to:
\[
  t = \frac{L}{3n}\log\frac{2d}{\delta_b} + \sqrt{\left(\frac{L}{3n}\log\frac{2d}{\delta_b}\right)^2 + \frac{2\sigma^2}{n}\log\frac{2d}{\delta_b}}.
\]
In our case, we will use the bound provided by Bennett's inequality applied to random Hermitian matrices, as it is simpler to derive the precision associated to a confidence level.
\begin{theorem}
  \label{thm:bernstein}
  Let $\mathbf{X}_1, \ldots, \mathbf{X}_n$ be $n$ i.i.d.\ random Hermitian matrices such that $\bbE \mathbf{X}_1 = \mathbf{0}$ and there exists $\sigma^2 > 0$ such that $\| \bbE\big[\mathbf{X}_1\big]^2 \| \leq \sigma^2$. In addition, let us assume that there exists $c > 0$ such that for any $q \geq 3$:
  \[
    \left\|\bbE\big[(\mathbf{X}_1)_+^q\big]\right\| \leq \frac{q!}{2} \sigma^2 c^{q - 2},
  \]
  where for any symmetric matrix $\mathbf{X}$,$(\mathbf{X})_+$ is the orthogonal projection of $\mathbf{X}$ onto the semidefinite positive cone.
  Then, for any $t > 0$, one has:
  \[
    \bbP \left( \left\| \sum_{i = 1}^n \mathbf{X}_i \right\| \geq \sqrt{2n\sigma^2t} + ct \right) \leq d e^{-t}.
  \]
\end{theorem}

The proof is very similar to Bernstein's: we need an intermediary result on the moment generating function, as stated in the following lemma.

\begin{lemma}
  \label{lma:bernstein}
  Let $\sigma^2, c > 0$ and let $\mathbf{X}$ be a random Hermitian matrix such that $\bbE \mathbf{X} = \mathbf{0}$ and $\| \bbE\big[\mathbf{X}^2\big] \| \leq \sigma^2$. In addition, we assume that for any $q \geq 3$,
  \[
    \left\| \bbE\big[(\mathbf{X})_+^q\big]\right\| \leq \frac{q!}{2} \sigma^2 c^{q - 2}.
  \]
  Then, for any $0 < t < 1 / c$, one has:
  \[
    \bbE \exp (t\mathbf{X}) \preceq \exp\left(\frac{t^2/2}{1 - ct} \bbE\big[ \mathbf{X}^2 \big]\right).
  \]
\end{lemma}
The proof of this Lemma is ommitted as it is very similar to Lemma~\ref{lma:bennett}. 

\begin{proof}[Proof of Theorem~\ref{thm:bernstein}]
  This proof can be directly adapted from Bernstein's using standard results on concentration (see \eg \cite{boucheron2013concentration} for details).
\end{proof}

\begin{proof}[Proof of Theorem~\ref{thm:big-result}]
  As mentionned in the beginning of this section, our goal is to bound $\|\mathbf{S}_n^{-1}(\mathbf{S}_n - \bbE\mathbf{S}_n)(\bbE\mathbf{S}_n)^{-1}\|$. Let us assume that the batch size $n$ satisfies:
\[
  n > \frac{2L \log d}{\lambda_{\mathrm{min}}(\bbE\mathbf{X}_1)}.
\]
Let $de^{-\frac{n\lambda_{\mathrm{min}}(\bbE\mathbf{X}_1)}{2L}} < \delta_h < 1$. Using the Chernoff's bound, we know that with probability at least $1 - \delta_h$, the following holds true:
\[
  \| \mathbf{S}_n^{-1} \| \leq \frac{\| (\bbE\mathbf{S}_n)^{-1} \|}{1 - \sqrt{\frac{2L}{n} \| (\bbE\mathbf{X}_1)^{-1} \| \log(d/\delta_h)}}.
\]
Similarly, let $0 < \delta_b < 1$; using Bennett's inequality, with probability at least $1 - \delta_b$, we have:
\[
  \|\mathbf{S}_n - \bbE \mathbf{S}_n\| \leq \frac{L}{3}\log\frac{d}{\delta_b} + \sqrt{2n\sigma^2\log\frac{d}{\delta_b}}.
\]
Combining these two results with a union bound leads to the following bound, with probability $1 - (\delta_b + \delta_h)$:
\begin{align}
  \|\mathbf{S}_n^{-1} - (\bbE\mathbf{S}_n)^{-1}\| &\leq \| (\bbE\mathbf{S}_n)^{-1} \|^2\frac{(L/3)\log(d/\delta_b) + \sqrt{2n\sigma^2\log(d/\delta_b)}}{1 - \sqrt{(2L/n) \| (\bbE\mathbf{X}_1)^{-1} \| \log(d/\delta_h)}} \notag \\
  &\leq \frac{1}{n^2} \| (\bbE\mathbf{X}_1)^{-1} \|^2\frac{(L/3)\log(d/\delta_b) + \sqrt{2n\sigma^2\log(d/\delta_b)}}{1 - \sqrt{(2L/n) \| (\bbE\mathbf{X}_1)^{-1} \| \log(d/\delta_h)}} 
\end{align}
In order to obtain a unified bound depending on one confidence parameter $1 - \delta$, one could optimize over $\delta_b$ and $\delta_h$, subject to $\delta_b + \delta_h = \delta$. This leads to a messy result and a negligible improvement. One can use simple values $\delta_b = \delta_h = \delta / 2$, so the overall bound becomes, with probability $1 - \delta$:
\[
  \|\mathbf{S}_n^{-1} - (\bbE\mathbf{S}_n)^{-1}\|
  \leq \frac{1}{n} \| (\bbE\mathbf{X}_1)^{-1} \|^2 \sqrt{\frac{2\sigma^2}{n}\log\left(\frac{2d}{\delta}\right)}
  \left(\frac{1 + \sqrt{(L^2/18\sigma^2 n)\log(2d/\delta)}}{1 - \sqrt{(2L/n)\|(\bbE\mathbf{X}_1)^{-1}\| \log(2d/\delta)}}\right).
\]
This can finally be formulated as follows:
\[
  \|\mathbf{S}_n^{-1} - (\bbE\mathbf{S}_n)^{-1}\|
  \leq \frac{1}{n} \| (\bbE\mathbf{X}_1)^{-1} \|^2\sqrt{\frac{2\sigma^2}{n}\log\left(\frac{2d}{\delta}\right)} + o\left(\frac{1}{n\sqrt{n}}\right).
\]
The final result yields using the fact that $\max_{\mathbf{x} \in \mathcal {X}} \|\mathbf{x}\|^2 = L$ and $f^{\star}_{G,n} = f^{\star}_{D,n} = \frac{d}{n}$.

The bound on $f_E(\mathbf{S}_n^{-1})$ is obtained similarly, only using the Hoeffding's result on minimum eigenvalue.
\end{proof}
\section{A refined approach of the dimension}
\label{sec:intdim}

\subsection{Intrinsic dimension}
\label{sec:intrinsic-dimension}

\begin{definition}[Intrinsic dimension]
  Let $d > 0$ and $\mathbf{S}_n \in \bbR^{d \times d}$ be a positive semi-definite matrix. The intrisic dimension of $\mathbf{S}_n$, denoted $\intdim(\mathbf{S}_n)$, is defined as follows:
  \[
    \intdim(\mathbf{S}_n) \triangleq \frac{\trace(\mathbf{S}_n)}{\|\mathbf{S}_n\|}.
  \]
  One always has $1 \leq \intdim(\mathbf{S}_n) \leq d$.
\end{definition}

As for the regular concentration proofs, we will need two useful lemmas: one for deriving a nicer upperbound from Markov's inequality and the other for forcing the intrinsic dimension into the bound.

\begin{lemma}
  \label{lma:intdim-markov}
  Let $\mathbf{Z} \in \bbR^d$ be a random Hermitian matrix and let $\psi : \bbR \rightarrow \bbR_+$ be non-decreasing and non-negative. Then, for any $t \in \bbR$ such that $\psi(t) > 0$, one has:
  \[
    \bbP(\| \mathbf{Z} \| \geq t) \leq \frac{1}{\psi(t)} \bbE \trace \left( \psi(\mathbf{Z}) \right).
  \]
\end{lemma}
\begin{proof}
  Let $t \in \bbR$. Since $\psi$ is non-decreasing, the event $\{ \| \mathbf{Z} \| \geq t \}$ contains $\{ \psi(\| \mathbf{Z} \|) \geq \psi(t) \}$. In addition, using the definition of $\psi(\mathbf{Z})$, one can easily notice that $\psi(\mathbf{Z}) \succeq \mathbf{0}$ and $\|\psi(\mathbf{Z})\| \geq \psi(\|\mathbf{Z}\|)$. Therefore, one can write:
  \[
    \bbP(\| \mathbf{Z} \| \geq t) \leq \bbP(\| \psi(\mathbf{Z}) \| \geq \psi(t)) \leq \bbP(\trace(\psi(\mathbf{Z})) \geq \psi(t)),
  \]
  where we used the fact that $\psi(\mathbf{Z}) \succeq \mathbf{0}$ in the rightmost inequality. Finally, one can conclude using Markov's inequality.
\end{proof}

\begin{lemma}
  \label{lma:intdim-cvx}
  Let $\varphi : \bbR \mapsto \bbR$ be a convex function and let $\mathbf{Z}$ be a positive semi-definite matrix. Then, one has:
  \[
    \trace \big(\varphi(\mathbf{Z})\big) \leq \intdim(\mathbf{Z}) \varphi\big(\|\mathbf{Z}\|\big) + (d - \intdim(\mathbf{Z})) \varphi(0).
  \]
  In particular, if $\varphi(0) = 0$, one has:
  \[
    \trace \big(\varphi(\mathbf{Z})\big) \leq \intdim(\mathbf{Z}) \varphi\big(\|\mathbf{Z}\|\big).
  \]
\end{lemma}
\begin{proof}
  Let $0 \leq x \leq \| \mathbf{Z} \|$. By convexity of $\varphi$, we can write:
  \[
    \varphi(x) \leq \varphi(0) + \big(\varphi(\| \mathbf{Z} \|) - \varphi(0)\big) \frac{x}{\|\mathbf{Z}\|}.
  \]
  Using Lemma~\ref{lma:trace-monotonicity}, we can extend the above inequality to $\mathbf{Z}$:
  \[
    \trace \big(\varphi(\mathbf{Z})\big) \leq \trace\big(\varphi(0)\mathbf{I}\big) + \frac{\varphi(\| \mathbf{Z} \|) - \varphi(0)}{\|\mathbf{Z}\|} \trace \big(\mathbf{Z}\big),
  \]
  which can be rearranged as follows:
  \[
    \trace \big(\varphi(\mathbf{Z})\big) \leq \intdim(\mathbf{Z}) \varphi\big(\|\mathbf{Z}\|\big) + (d - \intdim(\mathbf{Z})) \varphi(0),
  \]
  and the result holds.
\end{proof}

Using the two previous lemmas, we can adapt the proof in Hoeffding in order to obtain an improved bound with the intrinsic dimension.
\begin{theorem}
  \label{thm:hoef-intdim}
  Let $L > 0$ and let $\mathbf{X}_1, \ldots, \mathbf{X}_n$ be i.i.d.\ random matrices such that $\mathbf{0} \preceq \mathbf{X}_1 \preceq L\mathbf{I}$. Let $\mathbf{S}_n = \sum_{i = 1}^n \mathbf{X}_i$. For any $0 < \varepsilon < 1$, one can upperbound $\| \mathbf{S}_n \|$ as follows:
  \[
    \bbP(\|\mathbf{S}_n\| \geq (1 + \varepsilon) \|\bbE \mathbf{S}_n \| )\leq
    2 \times \intdim(\mathbf{Z}) \left(\frac{e^{\varepsilon}}{(1 + \varepsilon)^{1 + \varepsilon}}\right)^{\frac{\|\bbE \mathbf{S}_n\|}{L}}.
  \]
\end{theorem}
\begin{proof}
  Let $t, \eta > 0$. Using Lemma~\ref{lma:intdim-markov} with $\psi : x \in \bbR \mapsto (e^{\eta x} - 1)_+$ yields:
  \begin{equation}
    \label{eq:intdim-first-markov}
    \bbP\left( \| \mathbf{S}_n \| \geq t \right) \leq \frac{1}{e^{\eta t} - 1}\bbE \trace \big(( \exp(\eta \mathbf{S}_n) - \mathbf{I})_+\big) = \frac{1}{e^{\eta t} - 1} \trace \left( \exp(\eta \mathbf{S}_n) - \mathbf{I} \right),
  \end{equation}
  where we used the fact that $\mathbf{S}_n \succeq \mathbf{0}$ implies $\exp(\eta \mathbf{S}_n) \succeq \mathbf{I}$. Let $0 \leq x \leq L$, by convexity of the exponential, one has:
  \[
    e^{\eta x} - 1 \leq \big(e^{\eta L} - 1\big) \frac{x}{L}.
  \]
  Once again, we can extend this result to $\mathbf{S}_n$ and obtain:
  \[
    \trace\big( \exp(\eta \mathbf{S}_n) - \mathbf{I} \big) \leq \trace\left(\frac{e^{\eta L} - 1}{L} \mathbf{S}_n\right).
  \]
  Taking the expectation and using the inequality $x \leq e^x - 1$ yields:
  \[
    \bbE\trace\big(\exp(\eta \mathbf{S}_n) - \mathbf{I}\big) \leq \trace\left(\frac{e^{\eta nL} - 1}{nL}\bbE\mathbf{S}_n\right) \leq \trace\left(\exp\left(\frac{e^{\eta nL} - 1}{nL}\bbE\mathbf{S}_n\right) - \mathbf{I} \right)
  \]
  We can now use Lemma~\ref{lma:intdim-cvx} with $\varphi : x \in \bbR \mapsto e^x - 1$ to obtain a bound depending on $\| \bbE \mathbf{S}_n \|$:
  \[
    \bbE\trace\big(\exp(\eta \mathbf{S}_n) - \mathbf{I}\big) \leq \trace\left(\exp\left(\frac{e^{\eta nL} - 1}{nL} \bbE\mathbf{S}_n\right) - \mathbf{I} \right) \leq \intdim(\bbE \mathbf{S}_n) \left(e^{\frac{e^{\eta nL} - 1}{nL} \| \bbE\mathbf{S}_n \|} - 1\right).
  \]
  Combining the previous inequality with~\eqref{eq:intdim-first-markov} yields:
  \[
    \bbP\left( \| \mathbf{S}_n \| \geq t \right) \leq \intdim(\bbE \mathbf{S}_n) \times  \frac{e^{\frac{e^{\eta nL} - 1}{nL} \| \bbE\mathbf{S}_n \|} - 1}{e^{\eta t} - 1} \leq \intdim(\bbE \mathbf{S}_n) \times  \frac{e^{\eta t}}{e^{\eta t} - 1} \cdot e^{-\eta t + \frac{e^{\eta nL} - 1}{nL} \| \bbE\mathbf{S}_n \|}.
  \]
  The remainder of the proof consists in bounding $e^{\eta t}/(e^{\eta t} - 1)$ by $2$ and the rightmost term as in the regular Hoeffding's proof (see \cite{tropp2015introduction} for additional details).
\end{proof}

There are two main differences between this version of the Hoeffding's bound and the regular one. First, there is a factor 2 with the intrinsic dimension. This a not necessarily a big deal as we can win on other aspects. Then, we have obtained a bound on the highest eigenvalue, but not on the lowest. This is due to the current definition of the intrinsic dimension: we use this limitation as a motivation for the refinement we propose in the next section.

\subsection{A refined approach of the intrinsic dimension}
\label{sec:intdim-plus}

\begin{definition}[Upper and lower intrinsic dimension]
  Let $d > 0$ and $\mathbf{S}_n \in \bbR^{d \times d}$ be a positive semi-definite matrix. The upper and lower intrisic dimensions of $\mathbf{S}_n$, denoted $\updim(\mathbf{S}_n)$ and $\lowdim(\mathbf{S}_n)$ respectively, are defined as follows:
  \[
    \left\{
      \begin{array}{rcl}
        \updim(\mathbf{S}_n) &\triangleq &\dfrac{\trace\big(\mathbf{S}_n - \lambda_{\mathrm{min}}(\mathbf{S}_n) \mathbf{I}\big)}{\|\mathbf{S}_n\| - \lambda_{\mathrm{min}}(\mathbf{S}_n)}\\[1em]
        \lowdim(\mathbf{S}_n) &\triangleq &\dfrac{\trace\big(\| \mathbf{S}_n \| \mathbf{I} - \mathbf{S}_n\big)}{\|\mathbf{S}_n\| - \lambda_{\mathrm{min}}(\mathbf{S}_n)} = d - \updim(\mathbf{S}_n).       
      \end{array}
    \right.
  \]
  One always has $1 \leq \updim(\mathbf{S}_n), \lowdim(\mathbf{S}_n) \leq d - 1$.
\end{definition}
  
This definition brings a different information about the matrix at stake: instead of renormalizing the trace using the spectral norm, we also shift it using the lowest eigenvalue. With these new quantities, we are able to formulate a refined version of Lemma~\ref{lma:intdim-cvx}.
\begin{lemma}
  \label{lma:updim-cvx}
  Let $\varphi : \bbR \mapsto \bbR$ be a convex function and let $\mathbf{Z}$ be a positive semi-definite matrix. Then, one has:
  \[
    \trace \big(\varphi(\mathbf{Z})\big) \leq \updim(\mathbf{Z}) \varphi\big(\|\mathbf{Z}\|\big) + \lowdim(\mathbf{Z}) \varphi(\lambda_{\mathrm{min}}(\mathbf{Z})).
  \]
  This bound is always tighter than the one with the intrinsic dimension, that is:
  \[
    \updim(\mathbf{Z}) \varphi\big(\|\mathbf{Z}\|\big) + \lowdim(\mathbf{Z}) \varphi(\lambda_{\mathrm{min}}(\mathbf{Z})) \leq \intdim(\mathbf{Z}) \varphi\big(\|\mathbf{Z}\|\big) + (d - \intdim(\mathbf{Z})) \varphi(0).
  \]
\end{lemma}
\begin{proof}
  To prove both assertions, we will show a more general bound, of the form:
  \[
    \trace \big(\varphi(\mathbf{Z})\big) \leq f(l),
  \]
  for $0 \leq l \leq \lambda_{\min}(\mathbf{Z})$ and then we will show that $f$ is non-increasing. Let $0 \leq l \leq \lambda_{\mathrm{min}}(\mathbf{Z})$ and let $l \leq x \leq \| \mathbf{Z} \|$. Using the convexity of $\varphi$, one can write:
  \[
    \varphi(x) \leq \varphi(l) + (\varphi(\|\mathbf{Z}\|) - \varphi(l)) \frac{x - l}{\|\mathbf{Z}\| - l} = \frac{x - l}{\|\mathbf{Z}\| - l} \varphi(\|\mathbf{Z}\|) + \frac{\|\mathbf{Z}\| - x}{\|\mathbf{Z}\| - l} \varphi(l).
  \]
  Using Lemma~\ref{lma:trace-monotonicity} again, we can extend the above inequality to $\mathbf{Z}$:
  \[
    \trace \big(\varphi(\mathbf{Z})\big) \leq  \frac{\trace(\mathbf{Z} - l\mathbf{I})}{\|\mathbf{Z}\| - l} \varphi(\|\mathbf{Z}\|) + \frac{\trace(\|\mathbf{Z}\|\mathbf{I} - \mathbf{Z})}{\|\mathbf{Z}\| - l} \varphi(l).
  \]
  It is immediate to see that taking $l = 0$ leads to Lemma~\ref{lma:intdim-cvx} and $l = \lambda_{\mathrm{min}}(\mathbf{Z})$ shows the first assertion of this lemma. The last assertion of the theorem just comes from the convexity of $\varphi$: when applying the convexity bound on two segments $\mathcal{I} \subseteq \mathcal{J}$, the bound on $\mathcal{I}$ is necessarily tighter than the bound on $\mathcal{J}$.
\end{proof}

\begin{theorem}
  \label{thm:chernoff-matrix-int}
  Let $\mathbf{X}_1, \ldots, \mathbf{X}_n$ be i.i.d.\ positive semidefinite random matrices, such that there exists $L > 0$ verifying $\mathbf{0} \preceq \mathbf{X}_1 \preceq L \mathbf{I}$. Let $\mathbf{S}_n$ be defined as:
  \[
    \mathbf{S}_n \triangleq \sum_{i = 1}^n \mathbf{X}_i.
  \]
  In addition, let $\kappa$ be the condition number of $\bbE\mathbf{S}_n$, that is
  \[
    \kappa \triangleq \frac{\| \bbE \mathbf{S}_n \|}{\lambda_{\mathrm{min}}(\bbE \mathbf{S}_n)}.
  \]
  Then, for any $0 < \varepsilon < 1$, one can lowerbound $\lambda_{\mathrm{min}}(\mathbf{S}_n)$ as follows:
  \[
    \bbP(\lambda_{\mathrm{min}}(\mathbf{S}_n) \leq (1 - \varepsilon) \lambda_{\mathrm{min}}(\bbE \mathbf{S}_n)) \leq
    \tilde{d}_{\mathrm{min}} \left(\frac{e^{-\varepsilon}}{(1 - \varepsilon)^{1 - \varepsilon}}\right)^{\frac{n\lambda_{\mathrm{min}}(\bbE \mathbf{X}_1)}{L}},
  \]
  where
  \[
    \tilde{d}_{\mathrm{min}} = \lowdim(\bbE\mathbf{S}_n)
    + \updim(\bbE\mathbf{S}_n) e^{-n\varepsilon \lambda_{\mathrm{min}}(\bbE \mathbf{X}_1)(\kappa -1)/L}.    
  \]
  Similarly, one can upperbound $\| \mathbf{S}_n \|$ as follows:
  \[
    \bbP(\|\mathbf{S}_n\| \geq (1 + \varepsilon) \|\bbE \mathbf{S}_n \| )\leq
    \tilde{d}_{\mathrm{max}} \left(\frac{e^{\varepsilon}}{(1 + \varepsilon)^{1 + \varepsilon}}\right)^{\frac{n\|\bbE \mathbf{X}_1\|}{L}},
  \]
  where
  \[
    \tilde{d}_{\mathrm{max}} = \updim(\bbE\mathbf{S}_n)
    + \lowdim(\bbE\mathbf{S}_n) e^{-n\varepsilon \|\bbE \mathbf{X}_1\|(1 - \kappa^{-1})/L}.    
  \]
\end{theorem}
\begin{proof}
  The beggining of the proof is similar to the regular Hoeffding's proof so we can directly write, for $t, \eta > 0$:
  \[
    \bbP\left( \| \mathbf{S}_n \| \geq t \right) \leq e^{-\eta t}\trace\exp\left(\frac{e^{\eta L} - 1}{L}\bbE\mathbf{S}_n\right)
  \]
  Using Lemma~\ref{lma:updim-cvx} with $\varphi : x \in \bbR \mapsto e^{g(\eta)x}$ and $g : \eta \mapsto L^{-1}(e^{\eta L} - 1)$ yields:
  \begin{align*}
    \trace\exp\left(g(\eta)\bbE\mathbf{S}_n\right)
    &\leq \updim(\bbE\mathbf{S}_n) e^{g(\eta) \| \bbE\mathbf{S}_n \|}
      + \lowdim(\bbE\mathbf{S}_n) e^{g(\eta) \lambda_\mathrm{min}( \bbE\mathbf{S}_n )}\\
    &= \left(\updim(\bbE\mathbf{S}_n)
      + \lowdim(\bbE\mathbf{S}_n) e^{g(\eta) (\lambda_\mathrm{min}( \bbE\mathbf{S}_n ) - \| \bbE\mathbf{S}_n \|)}\right) e^{g(\eta) \| \bbE\mathbf{S}_n \|}.
  \end{align*}
  Using the same value for $\eta$ than in regular Hoeffding's proof thus yields:
  \[
    \bbP\left( \| \mathbf{S}_n \| \geq (1 + \varepsilon) \|\bbE\mathbf{S}_n\| \right) \leq \tilde{d} \left( \frac{e^\varepsilon}{(1 + \varepsilon)^{1 + \varepsilon}}\right)^{\|\bbE \mathbf{S}_n \|/ L}
  \]
  where
  \[
    \tilde{d} = \updim(\bbE\mathbf{S}_n)
    + \lowdim(\bbE\mathbf{S}_n) e^{g(\eta) (\lambda_\mathrm{min}( \bbE\mathbf{S}_n ) - \| \bbE\mathbf{S}_n \|)}.
  \]
  Pluging the value $\eta = L^{-1} \log(1 + \varepsilon)$ into $\tilde{d}$'s expression yields the result.
  The result on $\lambda_{\mathrm{min}}(\mathbf{S}_n)$ is very similar:
  \[
    \bbP\left( \lambda_{\mathrm{min}}(\mathbf{S}_n) \leq t \right) = \bbP\left( \|-\mathbf{S}_n)\| \geq -t \right) \leq e^{\eta t}\trace\exp\left(\frac{e^{-\eta L} - 1}{L}\bbE\mathbf{S}_n\right).
  \]
  Using the same reasoning, one obtains:
  \[
    \bbP\left( \lambda_{\mathrm{min}}(\mathbf{S}_n ) \leq (1 - \varepsilon) \|\bbE\mathbf{S}_n\| \right) \leq \tilde{d} \left( \frac{e^{-\varepsilon}}{(1 - \varepsilon)^{1 - \varepsilon}}\right)^{\lambda_{\min}(\bbE \mathbf{S}_n )/ L}
  \]
  where
  \[
    \tilde{d} = \lowdim(\bbE\mathbf{S}_n)
    + \updim(\bbE\mathbf{S}_n) e^{g(-\eta) (\| \bbE\mathbf{S}_n \| - \lambda_\mathrm{min}( \bbE\mathbf{S}_n ))},
  \]
  which proves the result since $\eta = -L^{-1} \log(1 - \varepsilon)$.
\end{proof}

\begin{theorem}
  Let $\mathbf{X}_1, \ldots, \mathbf{X}_n$ be $n$ i.i.d.\ random symmetric matrices such that $\bbE \mathbf{X}_1 = \mathbf{0}$ and there exists $L > 0$ such that $\| \mathbf{X}_1 \| \leq L$, almost surely. Let $\mathbf{S}_n \triangleq \sum_{i = 1}^n \mathbf{X}_i$. Let $\mathbf{V}$ be the covariance matrix of $\mathbf{X}_1$, that is $\mathbf{V} \triangleq \bbE\big[\mathbf{X}_1^2] - \mathbf{M}(\mu^{\star}_G)^2$ and let $\kappa$ be its condition number. Let $t$ verifying
  \[
    \frac{3n\|\mathbf{V}\|^2}{L} > t > \sqrt{n}\|\mathbf{V}\|^2 + \frac{L}{3\sqrt{n}}.
  \]

  Then, one has:
  \[
    \bbP \left( \left\| \mathbf{S}_n \right\| \geq \sqrt{t} \right) \leq \tilde{d} e^{\displaystyle-\frac{t^2}{4\|\mathbf{V}\|^2}},
  \]
  where
  \[
    \tilde{d} = \updim(\mathbf{V}) + \lowdim(\mathbf{V}) e^{-\frac{n}{16}(1 - \kappa^{-1})}.
  \]
\end{theorem}
\begin{proof}
  One can combine reasonings of Bernstein's regular proof with the proof of Hoeffding's with $\updim$ and $\lowdim$ to obtain:
  \[
    \bbP(\| \mathbf{S}_n \| \geq \sqrt{n} t) \leq \updim(\mathbf{V}) e^{-\frac{t^2/2}{\sigma^2 + Lt/3n}} + \lowdim(\mathbf{V}) e^{-\frac{t^2/2}{\sigma^2 + Lt/3n} (2 - \kappa^{-1})}.
  \]
  Let us assume that
  \[
    \frac{3n\sigma^2}{L} > t > \sqrt{n}\sigma^2 + \frac{L}{3\sqrt{n}}.
  \]
  Then, the previous result can be bounded as follows:
  \begin{align*}
    \bbP(\| \mathbf{S}_n \| \geq \sqrt{n} t)
    & \leq \left(\updim(\mathbf{V}) + \lowdim(\mathbf{V}) e^{-\frac{t^2}{4\sigma^2}(1 - \kappa^{-1})}\right) e^{-\frac{t^2}{4\sigma^2}}\\
    & \leq \left(\updim(\mathbf{V}) + \lowdim(\mathbf{V}) e^{-\frac{n}{16}(1 - \kappa^{-1})}\right) e^{-\frac{t^2}{4\sigma^2}},
  \end{align*}
  and the result holds.
\end{proof}

\section{Introduction to best arm identification in linear bandits}
\label{sec:mab-linear}

Let $d > 0$ and $\mathcal{X} \subseteq \bbR^d$ a subset of $\bbR^d$, corresponding to the bandit arms. The linear bandit setting assumes that the conditional distribution of the rewards given the arm follows a linear model: there exists an unknown parameter $\theta_{\star} \in \bbR^d$ such that the reward $r(\mathbf{x})$ associated to any action $\mathbf{x} \in \mathcal{X}$ is of the form
\[
  r(\mathbf{x}) = \theta_{\star}^{\top} \mathbf{x} + \epsilon,
\]
where $\epsilon$ is a $R$-subgaussian noise independent from $\mathbf{x}$. This linear structure implies that some information is shared between arms through the parameter $\theta_{\star}$: an action-reward pair $(\mathbf{x}, r(\mathbf{x}))$ gives information about $\theta_{\star}$ and thus about the reward distributions of the other actions. This makes this setting very different to the classical multi-armed bandit setting where the reward distributions of each action are assumed to be independent. Whereas multi-armed bandit algorithms mainly focus on the estimation of the mean reward of each action, linear bandit algorithms are mainly interested in the estimation of the parameter $\theta_{\star}$.

Depending on the context, the goal of a bandit algorithm can either be to maximize the cumulated reward (the sum of the rewards collected over several iterations) or to find the arm maximizing the reward, referred to as \emph{best arm identification} or \emph{pure exploration}.

We here focus on best arm identification in linear bandits whose objective is to find the arm $\mathbf{x}_{\star}$ maximizing the average reward:
\begin{equation*}
\mathbf{x}_{\star} = \argmax_{\mathbf{x} \in \mathcal{X}} \theta_{\star}^{\top}\mathbf{x} \, .
\end{equation*}
As the parameter $\theta_{\star}$ is unknown the aim is to design a strategy that will sequentially choose $t$ actions $\mathbf{x}_1, \ldots, \mathbf{x}_t \in \mathcal{X}$ and collect their associated rewards $r_i = \theta_{\star}^{\top} \mathbf{x}_i + \epsilon_i$, $1 \leq i \leq t$, where $\epsilon_1, \ldots, \epsilon_t$ are independent realizations of $\epsilon$, to obtain an estimate $\hat{\theta}_t$ of $\theta_{\star}$.
To find the best arm, the estimated prediction $\hat{\theta}_t^{\top}\mathbf{x}$ should be close to the real prediction $\theta_{\star}^{\top}\mathbf{x}$ for all $\mathbf{x} \in \mathcal{X}$. More precisely, rather than the reward prediction of an action itself, we are interested in comparing the predictions of each pair of arms. We thus want $\vert (\hat{\theta}_t - \theta_{\star})^{\top}(\mathbf{x} - \mathbf{x}') \vert$ to be small.

\begin{remark}
Note that compared to the multi-armed bandit case where known suboptimal arms are no longer played, the situation is different for the best arm identification case as playing suboptimal arms might give information about the parameter $\theta_{\star}$ and improve the discrimination of the unknown $\mathbf{x}_{\star}$ with other arms.
\end{remark}

Most of the designed strategies for best arm identification in linear bandits \citep{SoareNIPS2014,XuAISTATS2018,TaoICML2018} have relied on two concentration inequalities giving high probability bounds on the prediction error $\vert(\hat{\theta}_t-\theta_{\star})^{\top}\mathbf{x} \vert$ of the regression estimator $\hat{\theta}_t$ obtained from a sequence of action-reward pairs. The first concentration inequality is only valid when the sequence of actions is fixed and hence cannot depend on the observed random rewards.
The authors of \citep{abbasi2011improved} derived a concentration inequality which holds when the sequence of actions is adaptive to the observed random rewards. However this concentration inequality offers a looser bound than the one given for fixed sequences. In \citep{SoareNIPS2014} strategies relying on the fixed sequence bound are developed whereas \citep{XuAISTATS2018} designed a fully adaptive algorithm based on the adaptive bound. These two concentration inequalities are detailed below.

Let $\hat{\theta}_t(\lambda)$ denote the ridge estimate of $\theta_{\star}$ with a $\ell_2$-penalty $\lambda$:
\[
  \hat{\theta}_t(\lambda) \triangleq \argmin_{\theta \in \bbR^d} \sum_{s = 1}^t (\theta^{\top} \mathbf{x}_s - r_s)^2 + \frac{\lambda}{2} \| \theta \|^2.
\]
The ridge estimate $\hat{\theta}_t(\lambda)$ can be expressed in closed form:
\[
  \hat{\theta}_t(\lambda) = \hat{\mathbf{A}}_t(\lambda)^{-1} \mathbf{X}_t^{\top} \mathbf{r}_t,
\]
where $\mathbf{X}_t^{\top} \triangleq (\mathbf{x}_1, \ldots, \mathbf{x}_t)$, $\mathbf{r}_t^{\top} \triangleq (r_1, \ldots, r_t)$ and $\hat{\mathbf{A}}_t(\lambda) \triangleq \mathbf{X}_t^{\top} \mathbf{X}_t + \lambda \mathbf{I}_d$.

\paragraph{Fixed design concentration inequality.}
We assume that there is a finite number of arms $\vert \mathcal{X} \vert = K$. If $\lambda = 0$ and $(\mathbf{x}_i)_{1 \leq i \leq \infty}$ is a fixed sequence of actions (independent of the random rewards $(\mathbf{r}_i)_{1 \leq i \leq \infty}$) we have the following concentration inequality \citep{SoareNIPS2014}: for all $\delta \in (0,1)$,
\begin{equation}
  \label{eq:fixed_bound_2}
  \mathbb{P}\left(\forall t \in \mathbb{N}, \forall \mathbf{x} \in \mathcal{X}, \vert \theta_{\star}^{\top} \mathbf{x} - \hat{\theta}_t^{\top}\mathbf{x}\vert \leq 2R\Vert \mathbf{x}\Vert_{\hat{\mathbf{A}}_t^{-1}}\sqrt{2\log\left(\frac{6t^2K}{\pi^2\delta}\right)}\right) \geq 1 - \delta \, .
\end{equation}

One may notice that this result holds over these directions by replacing $K$ by $K^2$ in the logarithmic term, as there are of the order of $K^2$ such directions.\footnote{It suffices to consider exactly $K(K-1)/2$ directions as the result is the same for $\mathbf{x} - \mathbf{x}'$ and $\mathbf{x}' - \mathbf{x}$.}

\paragraph{Adaptive design concentration inequality.}
When the sequence of actions is chosen adaptively of the history, \ie for all $i \in \mathbb{N}$, $\mathbf{x}_i$ is allowed to depend on $(\mathbf{x}_1, r_1, \ldots, \mathbf{x}_{t-1}, r_{t-1})$, we need to rely on a result established by \citep{abbasi2011improved}: if $\lambda > 0$ and $\Vert \mathbf{x}_i \Vert \leq L$ for all $i$ then for all $\delta \in (0, 1)$ and all $\mathbf{x} \in \bbR^d$,

\begin{equation}
  \label{ineq:adaptive-bound}
  \mathbb{P}\left(\vert \theta_{\star}^{\top} \mathbf{x} - \hat{\theta}_t(\lambda)^{\top}\mathbf{x}\vert \leq \Vert \mathbf{x}\Vert_{\hat{\mathbf{A}}_t(\lambda)^{-1}}\left(R\sqrt{d\log\left(\frac{1 + tL^2/\lambda}{\delta}\right)} + \sqrt{\lambda}\Vert \theta_{\star}\Vert\right)\right) \geq 1 - \delta \, .
\end{equation}

The reader can refer to \citep[Appendix B]{abbasi2011improved} for the proof of this result. The main difference with \eqref{eq:fixed_bound_2} is the presence of an extra $\sqrt{d}$ factor which cannot be removed and which makes adaptive algorithms suffer more from the dimension than fixed design strategies (see \citep[Chapter 20]{LattimoreSzepesvariBanditsBook} for a more complete discussion on this aspect). We now omit the dependence of $\hat{\mathbf{A}}_t(\lambda)^{-1}$ in $\lambda$ when it is not relevant for the purpose of the discussion.

Whichever concentration inequality is used, the bound on the prediction error in a direction $\mathbf{y} = \mathbf{x} - \mathbf{x}'$, $\mathbf{x}, \mathbf{x}' \in \mathcal{X}$ depends on the matrix norm $\Vert \mathbf{y} \Vert_{\hat{\mathbf{A}}_t^{-1}}$. The goal of a strategy for the problem of best arm identification in linear bandits as formulated in \citep{SoareNIPS2014} is to choose a sequence of actions that reduces this matrix norm as fast as possible for all directions $\mathbf{y}$ so as to reduce the prediction error and be able to identify the best arm. This approach thus leads to the following optimization problem:
\begin{equation}
  (\mathbf{x}_1, \ldots, \mathbf{x}_B) \in \argmin_{\mathbf{x}_1, \ldots, \mathbf{x}_B \in \mathcal{X}} \max_{\mathbf{y} \in \mathcal{Y}} \,\mathbf{y}^{\top}\left(\sum_{i = 1}^t \mathbf{x}_i\mathbf{x}_i^{\top}\right)^{-1}\mathbf{y}.
\end{equation}

If one upper bounds $\Vert \mathbf{y} \Vert_{\hat{\mathbf{A}}_t^{-1}}$ by $2\Vert \mathbf{x} \Vert_{\hat{\mathbf{A}}_t^{-1}}$ we finally obtain the G-optimal design.

\section{Details on experiment setting and comments}
\label{sec:experiment_details}

For the randomized strategies we use the \emph{cvxopt} python package \citep{cvxopt} to compute the solution of the semi-definite program associated to the E-optimal design and compute the solution of the convex relaxation of the D-optimal design problem. We recall that as the relaxed G-optimal design problem is equivalent to the relaxed D-optimal design problem we can use the solution of the latter for the former.
Finally, for the greedy implementations of E and G-optimal design, when there are ties between several samples at a given iteration we uniformly select one at random.

\subsection{Randomized strategy versus greedy strategy for E-optimal design}

We recall here that the goal of E-optimal design is to choose experiments maximizing $\lambda_{\mathrm{min}}(\sum_{k=1}^K n_k \mathbf{x}_k \mathbf{x}_k^{\top})$. We generate a pool of experiments in $\bbR^d$ made of $K$ independent and identically distributed realizations of a standard Gaussian random variable. Figure \ref{e_random_vs_greedy_iterations} shows the performance of the randomized and greedy strategies against the number of selected samples $n$ when $K=500$ and $d=10$. For very small numbers of selected experiments the performances of the different strategies are equivalent but as the number of experiments increases the randomized E-optimal design outperforms the greedy strategy. Figure \ref{e_random_vs_greedy_dimension} shows the performance of the strategies against the dimension $d$ when $K=500$ and the number of selected experiments $n$ is fixed to 500. For small dimensions the randomized E-optimal design achieves a better performance but its superiority decreases when the dimension increases.
For both settings we also plot the performance of the random strategy that selects experiments uniformly at random. Furthermore, the results are averaged over 100 random seeds controlling the generation of the dataset as well as the random sampling of the experiments.

\subsection{Application of randomized G-optimal design to best arm identification in linear bandits}

We now compare the randomized G-optimal design with the greedy implementation that has been used for the problem of best arm identification in linear bandits. We note that the objective of this experiment is not to achieve state-of-the-art results for best arm identification in linear bandits but rather to show that the randomized strategy while being easy to implement achieves comparable results as the ones obtained with the greedy strategy.

The underlying model of a linear bandit is the same as the one presented in Section \ref{sec:preliminaries}: the relationship between the experiments $\mathbf{x}$, referred to as \emph{arms} in the bandit literature, and their associated measurements $y$ is assumed to be linear. The goal of best arm identification (see \eg \citep{SoareNIPS2014,TaoICML2018,XuAISTATS2018}) is to find the arm with maximum linear response among a finite set of arms. We focus on the case where one wants to solve this task with a minimum number of trials for a given confidence level. The core idea of most of the developed strategies is to sequentially choose arms so as to minimize a confidence bound on the prediction error of the linear response. Indeed, the sooner we become confident about the predicted response of each arm the sooner we can identify the best one with high probability.

One would like to take advantage of the past responses $y$ when choosing future arms. However the confidence bound that is available for this adaptive setting has a worse dependence on the dimension $d$ than the confidence bound available for fixed sequences of arm \citep{abbasi2011improved}. The confidence bound for fixed sequences can be stated as follows: for all $\delta \in (0,1)$, with probability at least $1-\delta$, for all $n \in \mathbb{N}$ and for all arms $\mathbf{x} \in \mathcal{X}$,
\begin{equation}
  \label{eq:fixed_bound}
\vert \theta_{\star}^{\top} \mathbf{x} - \hat{\theta}_n^{\top}\mathbf{x}\vert \leq 2c\Vert \mathbf{x}\Vert_{\mathbf{\Sigma}_D^{-1}}\sqrt{\log\left(\frac{6t^2\vert \mathcal{X}\vert}{\pi^2\delta}\right)} \, ,
\end{equation}
where $\hat{\theta}_n$ is the OLS estimator obtained with $n$ samples, $c$ is a constant depending on the variance of the Gaussian noise and $\mathbf{\Sigma}_D = \sum_{k=1}^K n_k \mathbf{x}_k \mathbf{x}_k^{\top}$.
It can be observed that designing a strategy that minimizes this confidence bound for all arms naturally leads to the G-optimal design optimization problem. The reader can refer to the supplementary material or \citep{SoareNIPS2014} for more details.

To compare the randomized G-optimal design with the greedy implementation used for best arm identification in \citep{SoareNIPS2014} we use the same setting as the one of the experiment presented in Section 6 of \citep{SoareNIPS2014}. More specifically we consider a set of $d+1$ arms in $\bbR^d$ where $d \geq 2$. This set is made of the $d$ vectors $(\mathbf{e}_1, \dots, \mathbf{e}_d)$ forming the canonical basis of $\mathbb{R}^d$ and one additional arm $\mathbf{x}_{d+1} = (\cos(\omega), \sin(\omega), 0, \dots, 0)^{\top}$ with $\omega = 0.1$. The true parameter $\theta_{\star}$ has all its coordinates equal to 0 except the first one which is set to 2. In this setting, the best arm, \ie the one with maximum linear response, is $\mathbf{e}_1$. One can also note that it is much harder to differentiate this arm from $\mathbf{x}_{d+1}$ than from the other arms. The noise of the linear model is a standard Gaussian random variable $\mathcal{N}(0, 1)$ and the confidence level in \eqref{eq:fixed_bound} is chosen equal to $\delta = 0.05$. We also use the same condition as in \citep{SoareNIPS2014} (equation (13) therein) to check when enough arms have been pulled to be able to identify the best arm with high probability. This condition naturally derives from the confidence bound \eqref{eq:fixed_bound}. As explained in Section \ref{sec:experimental-design} the greedy implementation does not work for the first iterations because the design matrix is singular. As in \citep{SoareNIPS2014} we thus initialize the procedure by choosing once each arm of the canonical basis. Although this would not be required for the randomized strategy as we could start by sampling a given number of experiments, we use the same initialization for the sake of fairness.

The number of samples required to find the best arm are shown in Figure \ref{expe_soare_random_vs_greedy} which summarizes the results obtained over 100 random seeds controlling the Gaussian noise of the linear model and the random selection of the experiments. One can see that the randomized G-optimal design, while being simple to use, achieves similar performances for low dimensions and even better performances on average than the greedy implementation of the G-optimal design as the dimension increases. We note that for all the random repetitions the best arm returned by both strategies is always $\mathbf{e}_1$.

\end{document}